\documentclass[12pt]{article}
\usepackage{amsfonts}
\usepackage{amsmath}
\usepackage{amssymb}
\usepackage{microtype}
\usepackage{lipsum}
\usepackage[mathscr]{eucal}
\usepackage{fullpage}
\usepackage{times}
\usepackage[usenames]{color}
\usepackage{hyperref}   

\def\eod{\vrule height 6pt width 5pt depth 0pt}
\newenvironment{proof}{\noindent {\bf Proof:} \hspace{.2em}}
{\hspace*{\fill}{\eod}}

\newcommand{\floor}[1]{\left\lfloor #1 \right\rfloor}

\newcommand{\etal}{\textit{et al.}}

\newcommand{\exc}{\mathrm{exc}}
\newcommand{\antiexc}{\mathrm{nexc}}

\newcommand{\wexc}{\mathrm{wkexc}}
\newcommand{\ExcSa}{\mathrm{EXC}}
\newcommand{\des}{\mathrm{des}}
\newcommand{\asc}{\mathrm{asc}}

\newcommand{\DescSet}{\mathrm{DES}}
\newcommand{\ascSet}{\mathrm{ASC}}
\newcommand{\inv}{\mathrm{inv}}

\newcommand{\SSS}{\mathfrak{S}}
\newcommand{\BB}{\mathfrak{B}}

\newcommand{\Der}{\mathfrak{SD}}

\newcommand{\AAA}{\mathcal{A}}

\newcommand{\GE}{ \mathsf{AExc}}

\newcommand{\Negs}{\mathsf{Negs}}

\newtheorem{theorem}{Theorem}

\newtheorem{corollary}[theorem]{Corollary}
\newtheorem{conjecture}[theorem]{Conjecture}

\newtheorem{remark}[theorem]{Remark}

\newtheorem{problem}[theorem]{Problem}
\newtheorem{definition}[theorem]{Definition}
\newtheorem{lemma}[theorem]{Lemma}

\newcommand{\comment}[1]{}

\begin{document}

	\title{Log-concavity of the Excedance Enumerators in 
		positive elements of Type A and Type B Coxeter Groups}
	
	\author{Hiranya Kishore Dey\\ 
		Department of Mathematics\\
		Indian Institute of Technology, Bombay\\
		Mumbai 400 076, India.\\
		email: hkdey@math.iitb.ac.in
	}
	
	\maketitle

\begin{abstract}

The classical Eulerian Numbers  $A_{n,k}$ are known to be 
log-concave. 
Let $P_{n,k}$ and 
$Q_{n,k}$ be the number of even and odd permutations with $k$ 
excedances. In this paper, we show that $P_{n,k}$ and 
$Q_{n,k}$ are  log-concave. For this, we introduce the notion of 
strong 
synchronisation and ratio-alternating which are motivated by 
the notion of synchronisation and ratio-dominance, introduced 
by Gross, 
Mansour, Tucker and Wang in 2014. 

We show similar results for Type B Coxeter Groups. We finish with some 
conjectures to emphasize the following: though strong synchronisation is  
stronger than log-concavity, many pairs of interesting combinatorial 
families of sequences seem to satisfy this property.

\end{abstract}	


\section{Introduction}
\label{sec:intro}

Log-concavity and unimodality are well-studied properties of 
combinatorial sequences. They often appear in various areas
 of mathematics such as combinatorics, probability and algebra.
  The papers of Br{\"a}nd{\'e}n
\cite{branden-unimodality_log_concavity}, Brenti (\cite{brentilogconcavepolya}, \cite{brentilogconcave})
and Stanley  \cite{stanleylogconcave} contain a wealth of 
information about various 
results on log-concavity. 
\begin{definition} 
	\label{defn:log_concave}
	A sequence $( a_k )_{k=0}^n$ is said 
	to be log-concave if for all $i=1,2,\ldots,n-1 $, we have
	$ a_{i} ^2 \geq a_{i-1} a_{i+1}.$
\end{definition}

\begin{definition} 
	\label{defn:unimodal}
	A  sequence $(a_k)_{k=0}^n$ is said to be
	unimodal if  there exists an index $0 \leq r \leq n$ such 
	that $a_0 \leq a_1 \leq  \ldots  \leq a_{r-1} \leq a_r \geq a_{r+1} \geq 
	\ldots \geq a_n$. 
\end{definition}

In this work, we will only deal with finite and  
non-negative sequences. Define a polynomial to be 
log-concave (and unimodal respectively), 
if the sequence of its coefficients is 
log-concave (and unimodal respectively). 
If a non-negative sequence  $(a_k)_{k=0}^n$ is 
log-concave and does not have any internal zero,
then there cannot be any $j$ such that 
$a_{j-1}>a_{j}<a_{j+1}$ and so the sequence $(a_k)_{k=0}^n$ must 
be unimodal. Many methods have been incorporated to establish the log-concavity of various combinatorial sequences. If the sequence satisfies some `nice' formula or recurrence, then by direct manipulation one can show log-concavity.
Another approach towards proving the log-concavity of a sequence
 is showing the real-rootedness of the associated  polynomial.  Combinatorial polynomials are
often real-rooted and 
Newton showed that real-rooted 
polynomials are log-concave. Thus this criterion directly solves many log-concavity related 
problems (see Petersen \cite[Chapter 4]{petersen-eulerian-nos-book}). Another interesting way of attacking a log-concavity problem
 is by directly giving a combinatorial proof. If $a_0, a_1, \dots, a_n$ is any sequence of non-negative integers for which a combinatorial
  meaning is known (that is, we have sets $S_0,S_1,\dots,S_n$ such
   that $|S_i|=a_i$), then constructing an  injection $\phi_k:S_{k-1} \times S_{k+1} \to S_k \times S_k$ yields a combinatorial proof of $a_k ^ 2 \geq a_{k+1}a_{k-1}$. One can take a look at \cite{saganinductivelogconcavity} where Sagan gave combinatorial proof of log-concavity of
    some combinatorial sequences. In this work, 
    we are interested in the following question:

Suppose we have two sequences $A=( a_k )_{k=0}^n$ and $B=( b_k )_{k=0}^n$. 
Let us define $S(A,B)$ to be 
the set of all sequences $C=( c_k )_{k=0}^n$ such that for each $k$, $c_k \in 
\{a_k,b_k\}$. $S(A,B)$ is actually the set of all $2^{n+1}$ sequences, which can be cooked up by using the two given sequences $A=( a_k )_{k=0}^n$ and $B=( b_k )_{k=0}^n$. The natural question, that comes to mind, is whether all the  sequences in $S(A,B)$ are log-concave or not. In this work, we investigate the above question for some interesting combinatorial pair of sequences.

For a positive integer $n$, let $[n] = \{1,2,\ldots,n \}$ and 
let $\SSS_n$ be the set of permutations on $[n]$.
For
$\pi = \pi_1,\pi_2,\ldots,\pi_n \in \SSS_n$, define its excedance
set as $\ExcSa(\pi) = \{ i \in [n]: \pi_i > i\}$  and its number of 
excedances as $\exc(\pi) = |\ExcSa(\pi)|$. Define  its number of 
antiexcedances as $\antiexc (\pi) = |\{ i \in [n]: \pi_i \leq i\}|$ 
and  inversions as $\inv(\pi) = | \{ 1 \leq i < j \leq n : \pi_i > \pi_j \} |$. Let $\DescSet(\pi) = \{i \in [n-1]: \pi_i > \pi_{i+1} \}$ and 
$\ascSet(\pi) = \{i \in [n-1]: \pi_i < \pi_{i+1} \}$
be its set of descents and  ascents respectively. 
Let $\des(\pi) = |\DescSet(\pi)|$ be its number of descents 
and $\asc(\pi) = |\ascSet(\pi)|$  be its number of ascents.  Let 
$\AAA_n \subseteq \SSS_n$ be the subset of
even permutations. Let $E_{n,k}$, $P_{n,k}$ and $Q_{n,k}$ be the number of 
permutations with $k$ excedances in $\SSS_n$, $\AAA_n$ and 
$\SSS_n-\AAA_n$ respectively.
Define

\vspace{-5 mm}
\begin{eqnarray}
\label{eqn:des}
A_n(t)  =  \sum_{\pi \in \SSS_n} t^{\des(\pi)} = 
\sum _{k=0}^{n-1}A_{n,k}t^k 
\mbox{ 
	\hspace{1 mm} 
	and 
	\hspace{1 mm}
} 
\GE_n(t)  =  \sum_{\pi \in \SSS_n} t^{\exc(\pi)} =
\sum _{k=0}^{n-1}E_{n,k}t^k, \\
\label{eqn:exc_even}
\GE^+_n(t)  =  \sum_{\pi \in \AAA_n} t^{\exc(\pi)} =
\sum _{k=0}^{n-1}P_{n,k}t^k
\mbox{ 
	\hspace{1 mm} 
	and 
	\hspace{1 mm}
} 
\GE^-_n(t)  =  \sum_{\pi \in \SSS_n - \AAA_n} t^{\exc(\pi)} =\sum _{k=0}^{n-1}Q_{n,k}t^k. 
\end{eqnarray}

It is a well known result of MacMahon \cite{macmahon-book} 
that both 
descents and excedances are equidistributed over $\SSS_n$. 
That is, for all positive integers $n$ and $0 \leq k \leq n-1$,   $A_{n,k}=E_{n,k} $. 
$A_n(t)$ is known to be real-rooted for all $n$ and 
hence the $A_{n,k}$s are log-concave. 
But the excedance enumerating 
polynomial over $\AAA_n$ and $\SSS_n- \AAA_n$ are 
not always real-rooted and hence log-concavity of  $P_{n,k}$ and $Q_{n,k}$
are not immediate. Moreover, we ask whether all the sequences that can be cooked up by using $P_{n,k}$ and $Q_{n,k}$ are log-concave. To answer this question, we introduce a notion of 
{\sl strong synchronisation}  which is influenced 
by the notion of {\sl synchronisation} as defined in the paper by Gross,
Mansour, Tucker and Wang \cite{mansourgross4authorcombinationoflogconcavity}. 
They defined the following. 

\begin{definition}
	\label{s}
	Two non-negative  sequences $A=(a_k)_{k=0}^n$ and $B=(b_k)_{k=0}^n$ 
	are said to be synchronised, 
	denoted as $A \sim{} B$  if both are log-concave and they satisfy 
	$a_{k-1}b_{k+1} \leq a_kb_k$ and 
	$a_{k+1}b_{k-1} \leq a_kb_k$ for all $1 \leq k \leq n-1$. 
\end{definition}
Here we generalise this further and define the following notion 
of strong synchronisation of two sequences.

\begin{definition}
	\label{def:ss}
	Two non-negative  sequences $A=(a_k)_{k=0}^n$ and $B=(b_k)_{k=0}^n$ 
	are said to be strongly synchronised, denoted as $A \approx B$ if the following holds for all $1 \leq k \leq n-1:$ 
		\begin{eqnarray}
	\label{eqn:defnstrongsynchro}
	(\min\{a_k,b_k\})^2 \geq \max\{a_{k+1},b_{k+1}\}.\max \{ a_{k-1},b_{k-1}\}.
	\end{eqnarray}
	 Clearly strong synchronisation 
	implies log-concavity of both sequences $A$ and $B$. 	
\end{definition}
For $n=5$, consider the following sequences $(P_{5,k})_{k=0}^4$ and $(Q_{5,k})_{k=0}^4$  :
\begin{eqnarray}
\label{strongsyneqexample}
P_{5} & = & (1,11,36,11,1) \nonumber,  \\
Q_{5} & = & (0,15,30,15,0) \nonumber.
\end{eqnarray}
It is easy to check that the sequences $(P_{5,k})_{k=0}^4$ and $(Q_{5,k})_{k=0}^4$
satisfy \eqref{eqn:defnstrongsynchro} and hence they are strongly synchronised.


Clearly, strong 
synchronisation is a much stronger property than 
synchronisation. Recall the sequences $P_{n,k}$ and $Q_{n,k}$ from \eqref{eqn:exc_even}.
One of our main result in this paper is 
the following:

\begin{theorem} 
	\label{thm:MainresultforevenlogconcaveTypeA}
	For  positive integers $n$, the sequences $P_n=(P_{n,k})_{k=0}^{n-1}$ 
	and $Q_n=(Q_{n,k})_{k=0}^{n-1}$
	are strongly synchronised and hence log-concave. 
\end{theorem}

The sum of two log-concave sequences need not be log-concave. But 
{Gross \etal} in \cite[Theorem 2.3]{mansourgross4authorcombinationoflogconcavity}
 showed that the 
sum of two synchronised sequences 
is log-concave. 
Hence Theorem \ref{thm:MainresultforevenlogconcaveTypeA} refines 
the log-concavity of $A_{n,k}$. We generalize our results to the case when excedances are summed over the elements with positive sign in Type B Coxeter Groups. Let $\BB_n$ be the set of permutations $\pi$ of $\{-n, -(n-1),
\ldots, -1, 1, 2, \ldots n\}$ 
satisfying $\pi(-i) = -\pi(i)$.  
$\BB_n$ is referred to as the hyperoctahedral group or the group of 
signed permutations on $[n]$ and $|\BB_n| = 2^n n!$. 
We use Brenti's \cite{brenti-q-eulerian-94} definition for Type B excedance and define the 
excedance polynomials of Type B.
There is a natural
notion of length in these groups and we get results when excedance 
enumeration is restricted
to elements with even length. For Type B Coxeter 
Groups, our main result is Theorem
\ref{thm:MainresultforevenlogconcaveTypeB}. 
\begin{theorem} 
	\label{thm:MainresultforevenlogconcaveTypeB}
	For positive integers $n$, the sequences  
	$P^B_n= (P^B_{n,k})_{k=0}^n$ and $Q^B_{n}=(Q^B_{n,k})_{k=0}^n$ are strongly 
	synchronised.  
\end{theorem}

We organize this paper as 
follows.
In Section \ref{section:basicpropertiesstrong},
we state and prove some basic properties of  strongly synchronised 
sequences. In Section 
\ref{section:strongsynchronisationoftypeAexcedance} we prove Theorem \ref{thm:MainresultforevenlogconcaveTypeA}. In Section 
\ref{sec:typeB} we prove Theorem \ref{thm:MainresultforevenlogconcaveTypeB}. In Section $\ref{sec:sagan}$, we modify a Theorem of Sagan to prove log-concavity of some combinatorial sequences directly.

\section{ Properties of strong synchronisation}
\label{section:basicpropertiesstrong}

Recall the sequence $A_{n,k}$ from  Section 
\ref{sec:intro}. Let $\Der_n$ be the set of derangements 
in $\SSS_n$ and let $D_{n,k}= |\{\pi \in \Der_n : \des(\pi)=k\}|$.
Consider the following sequences:
\begin{eqnarray}
\label{seqexample}
(A_{6,0},A_{6,1},A_{6,2},A_{6,3},A_{6,4},A_{6,5}) =  
(1,57,302,302,57,1) \nonumber  \\
(D_{6,0},D_{6,1},D_{6,2},D_{6,3},D_{6,4},D_{6,5}) =  
(0,16,104,120,24,1) \nonumber
\end{eqnarray}
Then it can be checked that these two sequences are 
synchronised in $k$  but not strongly synchronised in $k$  as 
$D_{6,1}^2=16^2 \leq 302=A_{6,0}A_{6,2}$.

By definition, clearly the  strong synchronisation relation 
is symmetric  but neither reflexive nor transitive. Consider the 
following 
example:
$$A=(1,4,5), \hspace{8 mm} B=(1,5,10), \hspace{8 mm} C=(1,6,25). $$ 
Here, $A \approx B$ and $B\approx C$ but $A$ and $C$ are not strongly 
synchronised. Moreover, note that $A$ and $C$ are not even 
synchronised. We also note that for any log-concave sequence $A$,
 we have $A \approx A$, but for $2$ different scalars $\lambda$ and $\mu$,  $\lambda A $ and $ \mu A$ may, or may not be strongly synchronised.  
The following is a useful result which gives a nice connection 
between the strong synchronisation of two sequences $A$ and $B$ and  
log-concavity of all the sequences in $S(A,B)$. 

\begin{theorem}
	\label{thm:strongsynchroandlog}
	Two non-negative sequences $A$ and $B$ are strongly synchronised 
	if and only if for all $C \in  S(A,B)$, $C$ is log-concave. 
\end{theorem}


\begin{proof}
	Let  $A$ and $B$ be strongly synchronised and $C \in
	 S(A,B)$. 
	Then $$c_k^2 \geq (\min\{a_k,b_k\})^2 \geq \max\{a_{k+1},b_{k+1}\}.\max \{ a_{k-1},b_{k-1}\} \geq c_{k+1}c_{k-1}.$$ Hence, $C$ is log-concave.
	
	Conversely, let $A$ and $B$ be two non-negative sequences such 
	that for all $C \in S(A,B)$, $C$ is log-concave.
	 We fix $k$ and 
	we need to show 
	\begin{align}
	\label{eqn:minmax}
	 (\min\{a_k,b_k\})^2 \geq \max\{a_{k+1},b_{k+1}\}.\max \{ a_{k-1},b_{k-1}\}.
	\end{align}
	Let us consider the following sequence $C$ with

	$$c_r=\begin{cases}
	\max\{a_{r},b_{r}\} & \text {if $r=k-1$ and $k+1$ }.\\
	\min\{a_{r},b_{r}\} & \text {if $r=k$ }.\\
	a_{r} & \text {elsewhere  }.
	\end{cases}$$
	
	Then log-concavity of $C$ ensures  \eqref{eqn:minmax}. For each $k$,
	 we can construct such a sequence $C$ whose log-concavity will ensure $\eqref{eqn:minmax}$ and hence, 
	we are done.
\end{proof}

\vspace{3 mm}
Next we consider $l$ non-negative sequences 
$T^1=(T^1_k)_{k=0}^{n},T^2=(T^2_k)_{k=0}^{n}, \ldots, T^l=(T^l_k)_{k=0}^{n}$. Let $S(T^1,T^2,\ldots,T^l)$  be 
the set of sequences $C=( c_k )_{k=0}^n$ such that for each $k$, 
$c_k \in \{ T^1_k,\ldots, T^l_k \}$. $S(T^1,T^2, \ldots, T^l)$ is essentially the set of all $l^{n+1}$ sequences, which can be made up from the given sequences $T^1,T^2, \ldots, T^l$. 

\begin{corollary}
\label{rem:strongsynchroandlog2}
Let $T^1,T^2, \ldots, T^l$ be $l$ non-negative sequences. Suppose $C$ is log-concave for all $C \in S(T^1,T^2,\ldots,T^l)$. Then for all $1 \leq i < j \leq l$, the sequences $T^i$ and $T^j$ are strongly synchronised. 
\end{corollary}

%

As $C$ is log-concave for all $C \in S(T^1,T^2,\ldots,T^l)$, $C$ is log-concave for all $C \in S(T^i,T^j)$. Hence, Corollary \ref{rem:strongsynchroandlog2} follows. Surprisingly,  
the converse of the above statement is not true.  
Consider the following three sequences:
$$T^1=(1,5,3), \hspace{8 mm} T^2=(7,6,3), \hspace{8 mm} T^3=(6,6,4). $$ 
Here, $T^1 \approx T^2$,  $T^2\approx T^3$ and $T^1 \approx T^3$. Consider the sequence $(7,5,4) \in S(T^1,T^2,T^3)$ which is not log-concave. 
It is easy to see that
log-concavity of a sequence $A=(a_k)_{k=0}^n$ with all intermediate terms positive
is equivalent to saying $a_ja_l \geq a_{j-i}a_{l+i}$ for all positive integers $j \leq l$ and $i \leq j$. Thus, Theorem \ref{thm:strongsynchroandlog} gives the following corollary which we
 will need later in Section \ref{section:strongsynchronisationoftypeAexcedance}.

\begin{corollary}
	\label{cor:strongsynchrnointerlacinglogconcave}
	Let, $A=(a_k)_{k=0}^n$ and 
	$B=(b_k)_{k=0}^n$ be 
    two sequences. The following are equivalent: 
    \begin{enumerate}
    \item
    \label{item_first_cor_8}  $A=(a_k)_{k=0}^n$ and 
	$B=(b_k)_{k=0}^n$ are strongly synchronised. 
	
	\item 
	\label{item_second_cor_8}
	 For all $j \leq l$ and for all positive
	integers $i$, we have
	\begin{equation}
	\label{eqn:corollaryjminusl}
    \min\{a_j,b_j\}.\min\{a_l,b_l\} \geq \max\{a_{j-i},b_{j-i}\}. \max\{a_{l+i},b_{l+i}\}.
	\end{equation}
	\end{enumerate}
\end{corollary}

\begin{proof}
We prove the forward implication at first. 
Let $A=(a_k)_{k=0}^n$ and $B=(b_k)_{k=0}^n$ be strongly synchronised. Thus, by Theorem \ref{thm:strongsynchroandlog}, any sequence $C \in S(A,B)$ is log-concave and hence, we are done.
The other direction follows
 by setting $j=l$ and $i=1$ in \eqref{eqn:corollaryjminusl}.
\end{proof}

\subsection{Ratio-Alternating Sequences}
\label{subsec:ratio-alt}

Gross 
\etal \hspace {.2 mm} introduced the {\sl ratio-dominance}  relation 
between two sequences in \cite{mansourgross4authorcombinationoflogconcavity}. Then they gave several results connecting 
 ratio-dominance and synchronisation. Motivated by those, 
we introduce a similar but different notion of {\sl ratio-alternating} defined as follows:

\begin{definition}
	\label{def:ratio-alt}
	Two non-negative sequences $A=(a_k)_{k=0}^n$ 
	and 
	$B=(b_k)_{k=0}^n$ are said to be ratio-alternating 
	if they
	satisfy either 
	
	\vspace{-4 mm}
	
	\begin{eqnarray}
	\label{eqn:ratioalt1}
	a_{2i} \leq b_{2i} \hspace{3 mm} \forall \hspace{2 mm}  
	0 \leq 2i \leq n  \hspace{3 mm} \mbox{and} \hspace{3 mm}  
	a_{2i+1} \geq b_{2i+1}
	\hspace{3 mm} \forall \hspace{2 mm}  0 \leq 2i+1 \leq n ,
	\end{eqnarray}
	
	\vspace{-2 mm}
	
	\
	or 
	
	\vspace{-6 mm}
	
	\begin{eqnarray}
	\label{eqn:ratioalt2}
	a_{2i} \geq b_{2i} \hspace{3 mm} \forall \hspace{2 mm}  
	0 \leq 2i \leq n  \hspace{3 mm} \mbox{and} \hspace{3 mm}  
	a_{2i+1} \leq b_{2i+1}
	\hspace{3 mm} \forall \hspace{2 mm}  0 \leq 2i+1 \leq n.
	\end{eqnarray}

\end{definition}

The relation  ratio-alternating  is reflexive and symmetric 
but not transitive.  Consider the following 
example: $$A=(1,5,7), \hspace{8 mm} B=(3,4,10), \hspace{8 mm} C=(2,6,8).$$ 
Here, $A$ and $B$ are ratio-alternating, $B$ and $C$ are also ratio-alternating
but $A$ and $C$ are not ratio-alternating. We need the following two definitions for the next Theorem.

\begin{definition}
\label{defn:even_log-concave}
A sequence $A=(a_k)_{k=0}^n$ is said 
to be  even log-concave (respectively, odd log-concave) if we have
$ a_{i} ^2 \geq a_{i-1} a_{i+1}$ for $i$ even (respectively, $i$ odd) and $1 \leq i \leq n-1$. 
\end{definition}


\vspace{- 4 mm}

\begin{theorem}
	\label{thm:relationbetweenratioaltandstrongsynchro}

	Let $A=( a_k )_{k=0}^n$ and $B=( b_k )_{k=0}^n$ be two 
	non-negative sequences which satisfy \eqref{eqn:ratioalt1}. 
	Then the following statements are equivalent: 
	\begin{enumerate} 
		\item  
		\label{itm:first}
		$A$ is even log-concave and $B$ is odd log-concave.
		\item
		\label{itm:third} $A$ and $B$ are strongly synchronised. 
	\end{enumerate}	
	In a similar manner, if $A$ and $B$ be two non-negative 
	sequences which satisfy \eqref{eqn:ratioalt2}, 
	then the following are equivalent: 
	\begin{enumerate} 
		\item  
		\label{itm:second}
		$A$ is odd log-concave and $B$ is even log-concave.
		\item
		\label{itm:fourth} $A$ and $B$ are strongly synchronised. 
	\end{enumerate}
	
\end{theorem}

\begin{proof}
	We consider only the first case, that is, when $A$ and $B$ satisfy \eqref{eqn:ratioalt1}. 
	It is easy to see that \ref{itm:third} implies 
	\ref{itm:first}. 
	We prove \ref{itm:first} implies 
	\ref{itm:third}. 
	Assume that $A$ is even log-concave and $B$ is odd log-concave. 
	Then for even $k$, we have
	\begin{align}
	(\min\{a_k,b_k\})^2= a_k^2 \geq a_{k+1}a_{k-1} = \max\{a_{k+1},b_{k+1}\}. \max\{a_{k-1},b_{k-1}\} . 
	\end{align}
	The first equality follows as $A$ and $B$ satisfy 
	\eqref{eqn:ratioalt1}. The inequality follows as $A$ is even log-concave. The last inequality also follows from 
	\eqref{eqn:ratioalt1}.
	Similarly for $k$ odd, we have 
	\begin{align}
    (\min\{a_k,b_k\})^2=b_k^2 \geq b_{k+1}b_{k-1} = \max\{a_{k+1},b_{k+1}\}. \max\{a_{k-1},b_{k-1}\}. 
	\end{align}
	Hence, $A$ and $B$ are strongly synchronised. 
The proof for the case when $A$ and $B$ satisfy \eqref{eqn:ratioalt2} is identical and hence omitted. 
\end{proof}

\section{Proof of Theorem \ref{thm:MainresultforevenlogconcaveTypeA}}
\label{section:strongsynchronisationoftypeAexcedance}

At first, we mention the following well-known recurrence, satisfied by the Eulerian numbers $A_{n,k}$. For reference, one can see \cite[Theorem 1.3]{petersen-eulerian-nos-book}.

\begin{theorem}
\label{thm:recurrence_classical_eulerian_recurrence}
 For positive integers $n,k$ with $n \geq 2$ and  $0 \leq k \leq n-1$, the numbers $A_{n,k}$ satisfy the following recurrence:
 \begin{eqnarray}
 \label{eqn:recurrence_classical_eulerian_recurrence}
  A_{n,k} & = &(k+1) A_{n-1,k} + (n-k) A_{n-1,k-1}, 
  \end{eqnarray}
 where $A_{1,0}=1$ and $A_{1,1}=0$. 
\end{theorem}

\begin{remark}
\label{remark:boundary_cases_type_a_3}
Setting $k=1$ in Equation \eqref{eqn:recurrence_classical_eulerian_recurrence}, we have $A_{n,1} = 2A_{n-1,1} + n-1$. Thus,  when $n \geq 3$, we have  $A_{n,1} \geq n+1$. 
\end{remark} 

Next, we recall the numbers $P_{n,k}$ and $Q_{n,k}$ from \eqref{eqn:exc_even}.   
The following identity involving $P_{n,k} $ 
and $Q_{n,k}$ was shown by Mantaci (see \cite{mantaci-thesis}, \cite{mantaci-jcta-93}).

\vspace{- 2 mm}

\begin{theorem}[Mantaci]
	\label{thm:differenceofoddandeven}
	For positive integers $n$ and $0 \leq k \leq n-1$, $P_{n,k}$ and $Q_{n,k}$ satisfy the 
	following:

	\begin{align}
	\label{eqn:recurrencesforevenandodd3}
	P_{n,k}-Q_{n,k} & =  (-1)^k\binom{n-1}{k}.
	\end{align}	
\end{theorem}

Later Sivasubramanian in \cite{siva-exc-det} gave a proof of 
Theorem \ref{thm:differenceofoddandeven} using determinant 
enumeration of suitably defined matrices. In \cite{mantaci-thesis} and \cite{mantaci-jcta-93}, Mantaci showed the following 
recurrences involving $P_{n,k} $ and $Q_{n,k}$. Though Mantaci did 
with anti-excedance enumerator, in an identical manner we can get the 
same recurrences for excedance enumerator. The recurrences are also 
shown by Dey and Sivasubramanian in \cite{siva-dey-gamma_positive_excedances_alt_group_aoc}.  



\begin{lemma}[Mantaci]
	\label{recurences of Mantaci}
	For positive integers $n$ and $0 \leq k \leq n-1$, 
	the coefficients $P_{n,k}$ and $Q_{n,k}$ satisfy the 
	following:
	\begin{align}
	\label{eqn:recurrencesforevenandodd1}
	P_{n,k} & =  kQ_{n-1,k} +(n-k)Q_{n-1,k-1} + P_{n-1,k}, \\
	\label{eqn:recurrencesforevenandodd2}
	Q_{n,k} & =  kP_{n-1,k} +(n-k)P_{n-1,k-1} + Q_{n-1,k}.
	\end{align}
\vspace{-1 mm} where $P_{1,0}=1$ and $Q_{1,0}=0$. 

\end{lemma}

\begin{remark}
\label{remark:boundary_cases_type_a}
From Lemma \ref{recurences of Mantaci}, it is easy to show that for any positive integer $n$,
we have $P_{n,0}=1$ and $Q_{n,0}=0$.
\end{remark}

From \eqref{eqn:recurrencesforevenandodd3}, we get the following 
corollary. 

\begin{corollary}
	\label{cor:ratioalternatingofevenexcedance}
	
	For positive integers $n$, the sequences $P_n=(P_{n,k})_{k=0}^{n-1}$
	and $Q_n=(Q_{n,k})_{k=0}^{n-1}$ are ratio-alternating. Moreover, they satisfy \eqref{eqn:ratioalt2}. 
\end{corollary}

\begin{proof}
By Equation \eqref{eqn:recurrencesforevenandodd3}, we have 
$P_{n,k} \geq Q_{n,k}$ for $k$ being even and 
$P_{n,k} \leq Q_{n,k}$ for $k$ being odd, completing the proof.   
\end{proof}


\begin{lemma}
	\label{lemma:t_i}
	
	For positive integers $n$ and $1 \leq k \leq n-2$, the coefficients 
	$P_{n,k}$ satisfy following relation:
	\begin{align}	
	\label{eq:pnkassumofti}
	P_{n,k} ^2- P_{n,k+1}P_{n,k-1} & =  \sum_{i=1}^{9} T_i (n,k), 
	\end{align}
	where 
	\vspace{-2 mm}
	\begin{align*}
	T_1(n,k) & =  (k^2-1)(Q_{n-1,k}^2-Q_{n-1,k+1}Q_{n-1,k-1}),\\
	T_2(n,k) & =  Q_{n-1,k}^2 + Q_{n-1,k-1}^2 -2Q_{n-1,k-1}Q_{n-1,k} ,\\
	T_3(n,k) & =  ((n-k)^2-1)(Q_{n-1,k-1}^2-Q_{n-1,k}Q_{n-1,k-2}),\\
	T_4(n,k) & =  P_{n-1,k}^2-P_{n-1,k+1}P_{n-1,k-1},\\
	T_5(n,k) & =  (k+1)(n-k+1)(Q_{n-1,k}Q_{n-1,k-1}- Q_{n-1,k+1}Q_{n-1,k-2}), \\
	T_6(n,k) & =   (k-1)(n-k-1)(Q_{n-1,k-1}Q_{n-1,k}- Q_{n-1,k-1}Q_{n-1,k}) , \\
	T_7(n,k) & =  (n-k-1)(Q_{n-1,k-1}P_{n-1,k}- Q_{n-1,k}P_{n-1,k-1}), \\
	T_8(n,k) & =  (n-k+1)(Q_{n-1,k-1}P_{n-1,k}- Q_{n-1,k-2}P_{n-1,k+1}), \\
	T_9(n,k) & =  2kQ_{n-1,k}P_{n-1,k}- (k+1)Q_{n-1,k+1}P_{n-1,k-1}-(k-1)Q_{n-1,k-1}P_{n-1,k+1}. 
	\end{align*}
	Similar identity holds for $Q_{n,k}$. 
\end{lemma}

\begin{proof}
	By \eqref{eqn:recurrencesforevenandodd1}, 
	\begin{eqnarray}
	\label{eq:pnksqaure}
	P_{n,k}^2 & = & k^2 Q_{n-1,k}^2 + (n-k)^2Q_{n-1,k-1}^2+P_{n-1,k}^2+2k(n-k)Q_{n-1,k}Q_{n-1,k-1} 
	\nonumber \\ 
	& & + 2kQ_{n-1,k}P_{n-1,k}+2(n-k)Q_{n-1,k-1}P_{n-1,k} , 
	\end{eqnarray}
	and 
	\begin{eqnarray}
	\label{eq:pnkplusoneminusone}
	P_{n,k+1}P_{n,k-1} & = & (k^2-1) Q_{n-1,k+1}Q_{n-1,k-1} +[(n-k)^2-1]Q_{n-1,k}Q_{n-1,k-2}\nonumber \\  & & + P_{n-1,k+1}
	P_{n-1,k-1}   + (k+1)(n-k+1)Q_{n-1,k+1}Q_{n-1,k-2} 
	\nonumber \\& & +(k-1)(n-k-1)Q_{n-1,k-1}Q_{n-1,k}    + (k+1)
	Q_{n-1,k+1}P_{n-1,k-1} +\nonumber \\ & & +
	(k-1)Q_{n-1,k-1}P_{n-1,k+1} +(n-k-1)Q_{n-1,k}P_{n-1,k-1} 
	\nonumber \\ & &   +(n-k+1)Q_{n-1,k-2}P_{n-1,k+1}. 
	\end{eqnarray}
	Subtracting \eqref{eq:pnkplusoneminusone} from \eqref{eq:pnksqaure} 
	and rearranging suitably, we get \eqref{eq:pnkassumofti}. 
\end{proof}

\vspace{3 mm}

\noindent
{\bf Proof of Theorem \ref{thm:MainresultforevenlogconcaveTypeA}: }
We use induction on $n$. The base cases when $n=2$ and $n=3$ are easy to be
verified. Assume that  $P_{n-1}$ and $Q_{n-1}$ are strongly
synchronised. 
	Hence, for all $k \in \mathbb{N}$ we have
	$P_{n-1,k}^2 \geq P_{n-1,k-1}P_{n-1,k+1}$ and $Q_{n-1,k}^2 \geq Q_{n-1,k-1}Q_{n-1,k+1}$. In particular,  for odd $k$ with 
	$1 \leq k \leq n-2$, we have
	$P_{n-1,k}^2 \geq P_{n-1,k-1}P_{n-1,k+1}$ and for 
	even $k$ with $1 \leq k \leq n-2$, we have
	$Q_{n-1,k}^2 \geq Q_{n-1,k-1}Q_{n-1,k+1}$. 
	
	\vspace{1 mm}
	 
	We show that for odd $k$ with 
	$1 \leq k \leq n-1$, $P_{n,k}^2 \geq P_{n,k-1}P_{n,k+1}$.  
	
	We first observe that
	\begin{enumerate}
		\item $ T_2(n,k) = (Q_{n-1,k}-Q_{n-1,k-1})^2  \geq 0 $
		\item  Next, we show that $T_1(n,k)$, $T_3(n,k)$, $T_4(n,k)$, $T_5(n,k)$, $T_{6}(n,k)$,  $T_{8}(n,k)$
		and $T_{9}(n,k)$ are non-negative using our inductive hypothesis. 
		Non-negativity of $T_1$, $T_3$ and $T_5$ follows from the 
		log-concavity of $Q_{n-1}$ while non-negativity of $T_4$ 
		follows from log-concavity of $P_{n-1} $. 
		As $P_{n-1}$ and $Q_{n-1}$ are strongly synchronised,
		we have $Q_{n-1,k}P_{n-1,k} \geq Q_{n-1,k+1}P_{n-1,k-1}$ and also $Q_{n-1,k}P_{n-1,k} \geq Q_{n-1,k-1}P_{n-1,k+1}$. Thus, $T_{9}$ is non-negative. 
		Non-negativity of $T_{8}$  follows from the 
		strong synchronisation of $P_{n-1}$ and $Q_{n-1}$ and Corollary \ref{cor:strongsynchrnointerlacinglogconcave}.
		Further, $T_6(n,k)=0$.   Hence we get that $T_4(n,k)+T_6(n,k)+T_{8}(n,k)+T_{9}(n,k) \geq 0$. The only negative term is $T_7(n,k).$ We will show that $T_1(n,k)+T_3(n,k)+T_{5}(n,k)+T_7(n,k) \geq 0$, that is the sum of the positive contributions of $T_1(n,k)$,  $T_3(n,k)$ and $T_5(n,k)$ will overkill the negative contribution of $T_7(n,k)$. 
	\end{enumerate}
	We need three further cases: (i) when $k=1$, (ii) when $k=n-1$, 
	(iii) when $3 \leq k \leq n-2$.  
	
	\vspace{2 mm}
	
	{\bf Case 1: When $k=1$} 
	
	\vspace{2 mm}

Setting $k=1$ and $k=2$ in Equation \eqref{eqn:recurrencesforevenandodd1}, we get 
	\begin{eqnarray}
	\label{eqn:kequals1CASE1}
	P_{n,1}^2-P_{n,2}P_{n,0}& = & (Q_{n-1,1}+P_{n-1,1})^2 - [2Q_{n-1,2}+(n-2)Q_{n-1,1}+P_{n-1,2}] \nonumber \\
	& = & [Q_{n-1,1}^2 - Q_{n-1,2}] + [Q_{n-1,1}P_{n-1,1}-Q_{n-1,2}] + [Q_{n-1,1}(P_{n-1,1}-(n-2))] \nonumber \\ 
	&  & + [P_{n-1,1}^2-P_{n-1,2}] 
	\end{eqnarray}
By induction, the sequences $P_{n-1}$ and $Q_{n-1}$ are strongly synchronised. Hence, $Q_{n-1,1}^2 - Q_{n-1,2}= Q_{n-1,1}^2 - Q_{n-1,2}P_{n-1,0} \geq 0$. Similarly, we can show that $Q_{n-1,1}P_{n-1,1}-Q_{n-1,2} \geq 0$ and $P_{n-1,1}^2-P_{n-1,2} \geq 0$. Further, by
Lemma \ref{recurences of Mantaci} and Remark \ref{remark:boundary_cases_type_a_3}, we have $P_{n-1,1}=Q_{n-2,1}+P_{n-2,1}=A_{n-2,1} \geq n-1$. Thus, we have $P_{n,1}^2-P_{n,2}P_{n,0} \geq 0$.

\vspace{2 mm}


	{\bf Case 2: When $k=n-1$}
	
	\vspace{-3 mm}
	
	\begin{align}
	P_{n,k}^2 \geq P_{n,k-1}P_{n,k+1}=0.
	\end{align}

	{\bf Case 3: When $3 \leq k \leq n-2$ and $k$ odd}

 \vspace{-1 mm}

	\begin{align}
	\label{eqn:T1}
	T_1(n,k) & =  (k^2-1)Q_{n-1,k}^2-(k^2-1)Q_{n-1,k+1}Q_{n-1,k-1 } 
	\nonumber \\ 
	& \geq  (k^2-1) Q_{n-1,k}^2- (k^2-1)P_{n-1,k}^2 \nonumber \\
	& =  (k^2-1)A_{n-1,k}(Q_{n-1,k}-P_{n-1,k}) \nonumber \\
	& =  (k^2-1)\binom{n-2}{k}A_{n-1,k} \nonumber \\
	& =  \frac{(k^2-1)(n-1-k)}{k}\binom{n-2}{k-1}A_{n-1,k} \nonumber  \\
	& \geq   (n-k-1)\binom{n-2}{k-1}Q_{n-1,k}.
	\end{align}
	
	The second line follows by induction. The third line uses 
	$A_{n-1,k}=P_{n-1,k}+Q_{n-1,k}$. The fourth line uses \eqref{eqn:recurrencesforevenandodd3}. The last line follows 
	from the fact that $k^2-1 \geq k $ when  $k \geq 2$.
	We observe that

	\begin{align}
	\label{eqn:tobeusedinT7}
	Q_{n-1,k+1}Q_{n-1,k-2} &= \frac{Q_{n-1,k+1}Q_{n-1,k}Q_{n-1,k-1}Q_{n-1,k-2}}{Q_{n-1,k}Q_{n-1,k-1}} \nonumber \\
	& \leq  \frac{Q_{n-1,k}P_{n-1,k}Q_{n-1,k-1}^2}{Q_{n-1,k}Q_{n-1,k-1}} =  P_{n-1,k}Q_{n-1,k-1}.
	\end{align}
	
	The first step is legitimate since  $Q_{n-1,k}$ and $Q_{n-1,k-1}$ are positive for $3 \leq k \leq n-2$.  
	The second step follows using strong synchronisation of the 
	sequences $P_{n-1,k}$ and $Q_{n-1,k}$. 
	\begin{align}
	\label{eqn:T_7}
	T_5(n,k) & =  (k+1)(n-k+1)(Q_{n-1,k}Q_{n-1,k-1}-Q_{n-1,k+1}Q_{n-1,k-2}) \nonumber \\
	& \geq  (k+1)(n-k+1)(Q_{n-1,k}Q_{n-1,k-1}-Q_{n-1,k-1}P_{n-1,k}) \nonumber \\
	& =  (k+1)(n-k+1)Q_{n-1,k-1}\binom{n-2}{k} 
	\end{align}	
	
	Here the second line uses \eqref{eqn:tobeusedinT7} and the 
	third line uses 
	line uses \eqref{eqn:recurrencesforevenandodd3}. 
	
	\vspace{- 5 mm} 
	
	\begin{align}
	\label{eqn:T_9}
	-T_7(n,k) & =  (n-k-1)(Q_{n-1,k}P_{n-1,k-1}- P_{n-1,k}Q_{n-1,k-1}) 
	\nonumber \\
	& = \displaystyle  (n-k-1)\bigg[(P_{n-1,k}+\binom{n-2}{k})(Q_{n-1,k-1}+\binom{n-2}{k-1})- P_{n-1,k}Q_{n-1,k-1}\bigg] 
	\nonumber \\
	& =  (n-k-1)\bigg[\binom{n-2}{k-1}Q_{n-1,k}+ \binom{n-2}{k}Q_{n-1,k-1}\bigg] \nonumber \\
	& =  (n-k-1)\binom{n-2}{k-1}Q_{n-1,k}+ (n-k-1)\binom{n-2}{k}Q_{n-1,k-1} \nonumber \\
	& \leq  T_1(n,k)+T_5(n,k).  
	\end{align}

	Here, both the second and the third line uses Theorem \ref{thm:differenceofoddandeven}. The fifth line 
	follows from \eqref{eqn:T1} and \eqref{eqn:T_7}.  
	Hence, when $3 \leq k \leq n-2$ and $k$ odd,  
	$T_1(n,k)+T_5(n,k)+T_7(n,k) \geq 0$.

	Thus for odd $k$ with $1 \leq k \leq n-1$, we have 
	$P_{n,k}^2 \geq P_{n,k-1}P_{n,k+1}$. In an identical manner 
	we can get $Q_{n,k}^2 \geq Q_{n,k-1}Q_{n,k+1}$ for 
	even $k$ with $1 \leq k \leq n-1$. So we proved that
	 $P_{n}$ is odd log-concave and $Q_n$ is even log-concave.
	  By Corollary \ref{cor:ratioalternatingofevenexcedance},
	  the sequences 
	   $P_n$ and $Q_n$ are
	ratio-alternating, hence by 
	Theorem \ref{thm:relationbetweenratioaltandstrongsynchro},
	 the sequences $P_{n}$ 
	and $Q_{n}$ are strongly synchronised. 
	{\hspace*{\fill}{\eod}}


\begin{corollary}
\label{cor:logconcavityofallbinarytypeA}
For positive integers $n$, 
all the sequences in $S(P_{n},Q_{n})$ are log-concave.   
\end{corollary}

\begin{proof}
By Theorem \ref{thm:MainresultforevenlogconcaveTypeA}, the sequences $P_n=(P_{n,k})_{k=0}^{n-1}$ and $Q_n=(Q_{n,k})_{k=0}^{n-1}$ are strongly synchronised. . Hence, by Theorem \ref{thm:strongsynchroandlog}, we get that all the sequences in $S(P_{n},Q_{n})$ are log-concave.  
\end{proof}

\section{Type B Coxeter groups}
\label{sec:typeB}

Let $\BB_n$ be the set of permutations $\pi$ of $\{-n, -(n-1),
\ldots, -1, 1, 2, \ldots n\}$ that satisfy $\pi(-i) = -\pi(i)$. 
For $\pi \in \BB_n$, for $1 \leq i \leq n$, we alternatively 
denote $\pi(i)$ as $\pi_i$.  
For $\pi \in \BB_n$, define $\Negs(\pi) = \{i : i > 0, \pi_i < 0 \}$ 
be the set of elements which occur with a negative sign.  
Define $\inv_B(\pi)=  | \{ 1 \leq i < j \leq n : \pi_i > \pi_j \} |+  
| \{ 1 \leq i < j \leq n : -\pi_i > \pi_j \}| +|\Negs(\pi)| $.  
Let $\BB^+_n \subseteq \BB_n$ denote the subset of 
elements having even $\inv_B()$ value and let $\BB_n^- = 
\BB_n \setminus \BB_n^+$.  
Following Brenti's definition of excedance from 
\cite{brenti-q-eulerian-94}, define $\exc_B(\pi)= |\{ i \in [n]: \pi_{|\pi(i)|} > \pi_ i\}|+
|\{ i \in [n]: \pi_i =- i\}|$ and define $\wexc_B(\pi)= |\{ i \in [n]: \pi_{|\pi(i)|} > \pi_ i\}|+
|\{ i \in [n]: \pi_i = i\}|$. 
For $\pi \in \BB_n $, let $ \pi_0=0$.  
We refer the reader to Petersen's book 
\cite[Chapter 13]{petersen-eulerian-nos-book}
for the following definition of type B descents. 
Define $\des_B(\pi)= |\{ i \in [0,1,2\ldots ,n-1]: \pi_i > \pi_{i+1}\}| $ 
and $\asc_B(\pi)= |\{ i \in [0,1,2\ldots ,n-1]: \pi_i <  \pi_{i+1}\}|$.
Let $B_{n,k}$, $B^+_{n,k} $ and $B^-_{n,k}$ denote the number of 
signed permutations with $k$ descents in $\BB_n$, $\BB_n^+$ and $\BB_n^-$ respectively. Let $E^B_{n,k}$, $P^B_{n,k} $ and $Q^B_{n,k}$ denote 
the number of signed
permutations with $k$ excedances in $\BB_n$, $\BB_n^+$ and $\BB_n^-$ respectively.  

Brenti in  \cite[Theorem 3.15]{brenti-q-eulerian-94} proved the
type B counterpart of MacMahon's theorem and showed that 
$B_{n,k} = E^B_{n,k}$.   Brenti proved the result by 
showing the following.

\begin{theorem}[Brenti]
	\label{thm:brenti_fft}
	For positive integers $n$, there exists a bijection 
	$h_n:\BB_n \mapsto \BB_n$ 
	such that $\asc_B(h_n(\pi))=\wexc_B(\pi)$ and $|\Negs((h_n(\pi))|= |\Negs(\pi)|$. 
\end{theorem}

Reiner in \cite[Theorem 3.2]{reiner-descents-weyl} proved the 
following:

\begin{theorem}[Reiner]
	\label{thm:Reinersigndescent}
	For positive integers $n$ and $0\leq k \leq n$, 
	$B^+_{n,k}$ and $B^-_{n,k}$ satisfy the following recurrence 
	relations:
	\vspace{-2 mm}
	\begin{align}
	\label{eqn:btypebinomialdescent}
	B^+_{n,k}-B^-_{n,k} & =  (-1)^k\binom{n}{k}.
	\end{align}
\end{theorem}

In \cite[Lemma 34]{siva-dey-gamma_positive_descents_alt_group_ejc},
Dey and Sivasubramanian proved the following recurrence between the coefficients $B^+_{n,k}$ and $B^-_{n,k}$.  

\begin{lemma} 
	\label{lemma:rec_exc_type_b}
	For positive integers $n$ and $1\leq k \leq n$, 
	$B^+_{n,k}$ and $B^-_{n,k}$ satisfy the following recurrence 
	relations:
	\begin{enumerate}
		\item $B^+_{n,k}= 2kB^-_{n-1,k} + (2n-2k+1)B^-_{n-1,k-1} + B^+_{n-1,k}, $
		\item $B^-_{n,k}= 2kB^+_{n-1,k} + (2n-2k+1)B^+_{n-1,k-1} + B^-_{n-1,k}, $
	\end{enumerate}
where $B_{1,0}^+=1$, $B_{1,1}^+=0$, $B_{1,0}^-=0$ and $B_{1,1}^-=1$ .
\end{lemma}

Sivasubramanian in \cite[Theorem 8]{siva-sgn_exc_hyp} enumerated 
the signed excedance polynomial over $\BB_n$ and showed
the following:
\vspace{-1 mm}
\begin{theorem}[Sivasubramanian]
	\label{thm:signedexcbtype}
	For positive integers $n$ and $0 \leq k \leq n$, 
	\begin{align}
	\label{eqn:btypebinomial}
	P^B_{n,k}-Q^B_{n,k} & =  (-1)^k\binom{n}{k}.
	\end{align}
	\vspace{- 3 mm}
	
\end{theorem} 

\vspace{- 3 mm}

By Theorem \ref{thm:brenti_fft}, we have $E^B_{n,k}=B_{n,k}$.  
From Theorem \ref{thm:Reinersigndescent} and Theorem 
\ref{thm:signedexcbtype}, 
we have $P^B_{n,k}-Q^B_{n,k}=
B^+_{n,k}-B^-_{n,k}$.  The coefficients $P^B_{n,k}$ and $Q^B_{n,k}$ 
also satisfy same initial conditions, hence we have $P^B_{n,k}=B^+_{n,k}$ 
and $Q^B_{n,k}=B^-_{n,k}$. Thus, we get the following lemma.  

\begin{lemma} 
	\label{lemma:rec_exc2_type_b}
	For positive integers $n$ and $1\leq k \leq n$, 
	$P^B_{n,k}$ and $Q^B_{n,k}$ satisfy the following 
	recurrence relations:
	\begin{enumerate}
		\item $P^B_{n,k}= 2kQ^B_{n-1,k} + (2n-2k+1)Q^B_{n-1,k-1} + P^B_{n-1,k}, $
		\item $Q^B_{n,k}= 2kP^B_{n-1,k} + (2n-2k+1)P^B_{n-1,k-1} + Q^B_{n-1,k}. $
	\end{enumerate}
where $P_{1,0}^B=1$, $P_{1,1}^B=0$, $Q_{1,0}^B=0$ and $Q_{1,1}^B=1$ .
\end{lemma}

%
%

From Theorem \ref{thm:signedexcbtype}, we immediately have the following 
corollary. 
\begin{corollary}
	\label{cor:ratioalternatingtypeb}
	For positive integers $n$, the sequences $P^B_{n,k}$ and 
	$Q^B_{n,k}$ are ratio-alternating. Moreover, they satisfy \eqref{eqn:ratioalt2}. 
\end{corollary}

We first prove Lemma \ref{lemma:ti_B} which is analogous to Lemma \ref{lemma:t_i}.

\begin{lemma}
	\label{lemma:ti_B}
	
	For positive integers $n$ and $1 \leq k \leq n-1$, the 
	coefficients  $P^B_{n,k}$ satisfy the following: 
	\vspace{- 2 mm}
	\begin{align*}	
	\label{eq:pnkassumoftiB}
	(P^B_{n,k}) ^2- P^B_{n,k+1}P^B_{n,k-1}  =  \sum_{i=1}^{9}T_i^B(n,k)
	\end{align*}
	where 
	\begin{align*}
	T_1^B(n,k) &  =  4(k^2-1)[(Q^B_{n-1,k})^2-Q^B_{n-1,k+1}Q^B_{n-1,k-1}],\\
	T_2^B(n,k) &  =  4(Q^B_{n-1,k})^2 + 4(Q^B_{n-1,k-1})^2 - 8Q_{n-1,k-1}Q_{n-1,k} ,\\
	T_3^B(n,k) &  =  ((2n-2k+1)^2-4)[(Q^B_{n-1,k-1})^2-Q^B_{n-1,k}Q^B_{n-1,k-2}],\\
	T_4^B(n,k)& =  (P^B_{n-1,k})^2-P^B_{n-1,k+1}P^B_{n-1,k-1},\\
	T_5^B(n,k) & =  2(k+1)(2n-2k+3)(Q^B_{n-1,k}Q^B_{n-1,k-1}- Q^B_{n-1,k+1}Q^B_{n-1,k-2}), \\
	T_6^B(n,k) & =   2(k-1)(2n-2k-1)(Q^B_{n-1,k-1}Q^B_{n-1,k}- Q^B_{n-1,k-1}Q^B_{n-1,k}), \\
	T_7^B(n,k) & =  (2n-2k-1)(Q^B_{n-1,k-1}P^B_{n-1,k}- Q^B_{n-1,k}P^B_{n-1,k-1}),\\
	T_8^B(n,k) & =  (2n-2k+3)(Q^B_{n-1,k-1}P^B_{n-1,k}- Q^B_{n-1,k-2}P^B_{n-1,k+1}), \\
	T_9^B(n,k) & =  4kQ^B_{n-1,k}P^B_{n-1,k}- 2(k+1)Q^B_{n-1,k+1}P^B_{n-1,k-1}-2(k-1)Q^B_{n-1,k-1}P^B_{n-1,k+1}. 
	\end{align*}
	
\end{lemma}

\begin{proof}
This proof follows by calculating $(P^B_{n,k}) ^2$ and $P^B_{n,k+1}P^B_{n,k-1}$ 
using the recurrences in Lemma  \ref{lemma:rec_exc_type_b}, 
as was done in the proof 
of Lemma \ref{lemma:ti_B} .
\end{proof}

Now we are in a position to prove our main result of this section. 

\vspace{2 mm}

\noindent
{\bf Proof of Theorem \ref{thm:MainresultforevenlogconcaveTypeB}:} 
	We prove this by induction along the same lines as in the 
	proof of Theorem \ref{thm:MainresultforevenlogconcaveTypeA}. 
	By induction assume that 
	$P^B_{n-1}$ and $Q^B_{n-1}$ are strongly synchronised 
	in $k$. 
	Hence, for
	$1 \leq k \leq n-1$ with $k$ odd, 
	$(P^B_{n-1,k})^2 \geq P^B_{n-1,k-1}P^B_{n-1,k+1}$ and for $1 \leq k \leq n-1$ 
	with $k$ even, 
	$(Q^B_{n-1,k})^2 \geq Q^B_{n-1,k-1}Q^B_{n-1,k+1}$. 
	Proceeding along the same line as in proof of Theorem \ref{thm:MainresultforevenlogconcaveTypeA}, we get that for 
	odd $k$ with $0 \leq k \leq n$, we have $(P^B_{n,k})^2 \geq P^B_{n,k-1}P^B_{n,k+1}$.
	 In an identical manner 
	we can get $(Q^B_{n,k})^2 \geq Q^B_{n,k-1}Q_{n,k+1}$ for 
	even $k$ with $0 \leq k \leq n$. Thus, the sequence $P^B_n$ 
	is odd log-concave and $Q^B_n$ is even log-concave. By Corollary \ref{cor:ratioalternatingofevenexcedance}, $P^B_{n}$ and 
	$Q^B_{n}$ are ratio-alternating, hence by 
	Theorem \ref{thm:relationbetweenratioaltandstrongsynchro}, 
	the sequences $P^B_{n}$ 
	and $Q^B_{n}$ are strongly synchronised.
{\hspace*{\fill}{\eod}}	 	

As the sequences $P_{n,k}$, $Q_{n,k}$, $P^B_{n,k}$ and 
$Q^B_{n,k}$ don't have any internal zeroes, hence log-concavity 
of those polynomials directly gives the following result. 
\begin{corollary}
	\label{thm:unimodality}
	For positive integers $n$, the sequences $P_{n,k}$, $Q_{n,k}$, 
	$P^B_{n,k}$ 
	and $Q^B_{n,k}$ are unimodal.
\end{corollary}

\section{Modification of Sagan's Theorem}
\label{sec:sagan}

Given a triangular array of non-negative integers, Sagan in \cite[Theorem 1]{saganinductivelogconcavity} gave the following condition which ensures 
that every row of the array is log-concave.

\begin{theorem}[Sagan]
	\label{thm:sagan}
	Suppose that for $n \geq 1$ and $0 \leq k \leq n$, 
	a non-negative integral sequence $t_{n,k}$ 
	satisfies the following triangular recurrence relation: $t_{n,k}=c_{n,k}t_{n-1,k-1}+d_{n,k}t_{n-1,k}$ where the 
	multiplicative coefficients $c_{n,k},d_{n,k}$ are all 
	non-negative integers and $t_{a,b}=0$ whenever $a < b$. 
	Suppose the following conditions hold: 
	(i) For each positive integer $n$, $c_{n,k}$ and $d_{n,k}$ are 
	log-concave in $k$.  
	(ii) $c_{n,k-1}d_{n,k+1}+c_{n,k+1}d_{n,k-1} \leq 2c_{n,k}d_{n,k}$ 
	for all $n \geq 1$ and $0 \leq k \leq n$. 
	Then, for each positive integer $n$, the sequence $t_{n,k}$ is 
	log-concave in $k$.
\end{theorem}

Looking at Theorem \ref{thm:sagan}, one may not consider this to be much of a labor-saving method, but working with the coefficient arrays in most of the cases is much simpler than working with the original ones. This theorem has some nice applications. This directly gives the log-concavity of binomial coefficients and Stirling Number of both kinds. But this does not 
directly prove the log-concavity of Eulerian Numbers as (ii) is not satisfied. 
Hence we modify Sagan's theorem. The proof of Theorem \ref{thm:modifiedSagan} goes along the same line as the original proof 
of Theorem \ref{thm:sagan} but we give it for completeness. 

\begin{theorem}
	\label{thm:modifiedSagan}
	Suppose that for $n \geq 1$ and $0 \leq k \leq n$, 
	a non-negative integral sequence $t_{n,k}$ 
	satisfies the following triangular recurrence relation: $t_{n,k}=c_{n,k}t_{n-1,k-1}+d_{n,k}t_{n-1,k}$ where 
	$c_{n,k},d_{n,k}$ are all non-negative integers and 
	$t_{a,b}=0$ whenever $a < b$. Suppose the following 
	conditions hold:\\
	(i) For each positive integers $n$, $c_{n,k}$ and $d_{n,k}$ are 
	log-concave in $k$.  \\
	(ii) $2 \sqrt{(c_{n,k}^2-c_{n,k+1}c_{n,k-1})
		(d_{n,k}^2-d_{n,k+1}d_{n,k-1}) }
	\geq c_{n,k-1}d_{n,k+1}+c_{n,k+1}d_{n,k-1}-2c_{n,k}d_{n,k}$ 
	for all $n \geq 1$ and all $0 \leq k \leq n$. \\
	Then, for 
	each positive integers $n$, the sequence $t_{n,k}$ is log-concave 
	in $k$.
\end{theorem}

\begin{proof}
	We prove this by induction. Assume $t_{n-1,k}$ to be log-concave 
	in $k$. 
	\begin{eqnarray*}
		t_{n,k}^2 - t_{n,k+1}t_{n,k-1}
		&= & c_{n,k}^2
		t_{n-1,k-1}^2 -c_{n,k+1}c_{n,k-1}t_{n-1,k}t_{n-1,k-2} + d_{n,k}^2
		t_{n-1,k}^2
		\\ & &  -d_{n,k+1}d_{n,k-1}t_{n-1,k+1}t_{n-1,k-1} +
		2c_{n,k}d_{n,k}t_{n-1,k-1}t_{n-1,k}
		\\ & & 
		-c_{n,k+1}d_{n,k-1}t_{n-1,k-1}t_{n-1,k}- c_{n,k-1}d_{n,k+1}t_{n-1,k-2}t_{n-1,k+1}
		\\& \geq  & (c_{n,k}^2 - c_{n,k+1}c_{n,k-1} )t_{n-1,k-1}^2 + 
		(d_{n,k}^2 - d_{n,k+1}d_{n,k-1} )t_{n-1,k}^2
		\\ & &
		+(2c_{n,k}d_{n,k}-c_{n,k-1}d_{n,k+1}-c_{n,k+1}d_{n,k-1})
		t_{n-1,k-1}t_{n-1,k}
	\end{eqnarray*}
	By A.M-G.M inequality $Ax^2+By^2  \geq 2 \sqrt{AB}xy $ whenever 
	$A$ and $B$ are non-negative. Let, $A=c_{n,k}^2 - c_{n,k+1}c_{n,k-1}$ 
	and $B=d_{n,k}^2 - d_{n,k+1}d_{n,k-1}$. Then $A$ and $B$ are 
	non-negative due to log-concavity of $c_{n,k}$ and $d_{n,k}$ 
	respectively. Here,
	\begin{eqnarray}
	2\sqrt{AB} & = & \sqrt{4(c_{n,k}^2-c_{n,k+1}c_{n,k-1})(d_{n,k}^2-d_{n,k+1}d_{n,k-1})} \nonumber \\
	& \geq &
	(c_{n,k-1}d_{n,k+1}+c_{n,k+1}d_{n,k-1}-2c_{n,k}d_{n,k}) \nonumber 
	\end{eqnarray}
	Hence, $t_{n,k}^2 - t_{n,k+1}t_{n,k-1}$ is non-negative 
	and so we are done. 
\end{proof}
\begin{remark}
	Note that, Theorem \ref{thm:modifiedSagan} also gives a necessary condition
	 to ensure that every row of a triangular array satisfying that condition
	  will be log-concave. But Theorem \ref{thm:modifiedSagan} is more general  because, if the multiplicative coeffcients 
	$c_{n,k}$ and $d_{n,k}$ satisfy (ii) of Theorem \ref{thm:sagan}, 
	then they certainly  satisfy (ii) of Theorem \ref{thm:modifiedSagan}. But, there are examples (all the examples in the next subsection) for which $c_{n,k}$ and $d_{n,k}$ satisfy  (ii) of Theorem \ref{thm:modifiedSagan} but don't satisfy 
	(ii) of Theorem \ref{thm:sagan}. Here also
	we work with the coefficient arrays instead of the original ones.  
\end{remark}

\subsection{Direct Applications of Modified Sagan's Theorem}

Here, we give some applications of Theorem \ref{thm:modifiedSagan}. At first, we consider some combinatorial sequences, whose log-concavity is already known using real-rootedness or some other tools. But here, we give direct proofs of log-concavity of those sequences. 

\begin{enumerate}
	\item {\bf Log-concavity of Eulerian Numbers:}
	\label{item: Log-concavity of Eulerian Numbers} 
	 Frobenius showed that the Eulerian polynomials $A_n(t)$ are real-rooted (for reference, one can see \cite[Chapter 4]{petersen-eulerian-nos-book}) and hence log-concave.
	 Theorem \ref{thm:modifiedSagan} immediately provides us an alternate proof of the log-concavity of 
	 Eulerian Numbers. We know from
	\cite[Theorem 1.3]{petersen-eulerian-nos-book} that $A_{n,k}$ 
	satisfy the following recurrence:
	$$A_{n,k}=(k+1)A_{n-1,k}+
	(n-k)A_{n-1,k-1}.$$
	 It is easy to see that both  $c_{n,k}=k+1$ and $d_{n,k}=(n-k)$ are log-concave in $k$. Further, 
	\begin{eqnarray*}
	2\sqrt{(c_{n,k}^2-c_{n,k+1}c_{n,k-1})(d_{n,k}^2-d_{n,k+1}d_{n,k-1})} \\ = 
	2 \geq 2= (c_{n,k-1}d_{n,k+1}+c_{n,k+1}d_{n,k-1}-2c_{n,k}d_{n,k}).
	\end{eqnarray*} 
	Thus, $c_{n,k}$ and $d_{n,k}$ satisfy the conditions of
	Theorem \ref{thm:modifiedSagan}. Hence $A_{n,k}$ is log-concave.

	\item {\bf Log-concavity of Type B Eulerian Numbers:}  Let us consider the  Eulerian polynomials of Type B: 
	$B_n(t)= \sum_{\pi \in \BB_n}t^{\des_B(\pi)} =
	\sum_{k=0}^n B_{n,k}t^k$. Brenti \cite{brenti-q-eulerian-94} showed that these polynomials are real-rooted and hence $B_{n,k}$s are log-concave. Theorem \ref{thm:modifiedSagan} gives another proof of log-concavity of the sequence $B_{n,k}$.
	From \cite{brenti-q-eulerian-94}, we get 
	that they satisfy the following recurrence:
	$$B_{n,k}=(2k+1)B_{n-1,k}+
	[2(n-k)+1]B_{n-1,k-1}.$$
	Taking 
	$c_{n,k}=2k+1$ and $d_{n,k}=2(n-k)+1$ works here as both 
	of them are log-concave and 
	\begin{eqnarray*}
	2\sqrt{(c_{n,k}^2-c_{n,k+1}c_{n,k-1})(d_{n,k}^2-d_{n,k+1}d_{n,k-1})} \\
	= 
	8 \geq 8  (c_{n,k-1}d_{n,k+1}+c_{n,k+1}d_{n,k-1}-2c_{n,k}d_{n,k}).
\end{eqnarray*} 
Thus, by Theorem \ref{thm:modifiedSagan}, they are also log-concave. 
	
	\item {\bf Log-concavity of Second order Eulerian Numbers:} Let $Q_n$ be the set of permutations of $\{1,1,2,2,\ldots,n,n\}$ 
	such that for all $i$, 
	entries between two occurences of $i$ are larger than $i$. For a permutation $\pi \in Q_n$,
	let $\DescSet(\pi) = \{i \in [2n-1]: \pi_i > \pi_{i+1} \}$ and 
	$\des(\pi) = |\DescSet(\pi)|$ be its number of descents. 
	Let $H_{n,k}= |\{\pi \in Q_n: \des(\pi)=k\}|$. These numbers $H_{n,k}$ are called as the second-order Eulerian Numbers. Bona in \cite{Bona-realzeroes} proved that the associated polynomials $H_n(t)= \sum_{k=0}^{n}H_{n,k}t^k$ are real-rooted which immediately gives log-concavity of $H_{n,k}$. Here we provide another proof of log-concavity of $H_{n,k}$. From \cite{haglund-visontai-stable-multivariate}, we get that these 
	coefficients $H_{n,k}$ satisfy the following recurrence: 
	$$H_{n,k}=kH_{n-1,k}+(2n-k)H_{n-1,k-1}.$$ Taking 
	$c_{n,k}=k$ and $d_{n,k}=2n-k$ and applying Theorem \ref{thm:modifiedSagan}, we immediately get an alternate proof of the log-concavity of 
	$H_{n,k}$ for all positive integers $n$.
	
	\end{enumerate}
	
	We now turn our attention to palindromic polynomials. A polynomial $f(t)= \sum_{i=0}^n a_it^i$ is said to be palindromic if $a_i=a_{n-i}$ for all $0 \leq i \leq \lfloor n/2 \rfloor . $ A palindromic polynomial $f(t)=\sum_{i=0}^n a_it^i$ is said to be gamma-positive if $f(t)=\sum _{i=0}^{\lfloor n/2 \rfloor} \gamma_{n,i} t^i(1+t)^{n-2i}$ with  $\gamma_{n,i} \geq 0$ for all $ 0 \leq i \leq \lfloor n/2 \rfloor$. 
	 One can see the survey paper of Athanasiadis  \cite{athanasiadis-survey-gamma-positivity} for a good reference on various gamma-positivity results. 

	\begin{enumerate}
	
	\item {\bf Log-concavity of the gamma-coefficients of Type A Eulerian polynomials:} Foata and Sch{\"u}tzenberger in \cite{foata-schutzenberger-eulerian} showed that  
	the Eulerian polynomials of Type A are
	gamma positive. 
	Let $T_{n,k}$ be the coefficient of $t^{2k}(1+t)^{n-1-2k}$ 
	in $A_n(t)$. Foata and Strehl in \cite{foata-strehl} gave a combinatorial interpretation of $T_{n,k}$. They
	proved that $T_{n,k}$ is actually the number of elements 
	in $S_n$ with $k$ descents and no double descents. From
   \cite[Theorem 8]{siva-dey-gamma_positive_descents_alt_group_ejc}
	we get that  these coefficients satisfy the following 
	recurrence: $$T_{n,k}=(k+1)T_{n-1,k}+(2n-4k)T_{n-1,k-1}.$$ 
	Let $c_{n,k}=(k+1)$ and $d_{n,k}=2n-4k$. Then, $c_{n,k}$ and $d_{n,k}$ satisfy the conditions of Theorem \ref{thm:modifiedSagan}. Thus, by Theorem \ref{thm:modifiedSagan}, the sequence $T_{n,k}$ is log-concave for any $n$.

	\item {\bf Log-concavity of gamma-coefficients of Type B Eulerian polynomials:} Chow in 
	\cite[Theorem 4.7]{chow-certain_combin_expansions_eulerian} 
	proved that the Type B Eulerian polynomials $B_n(t)= \sum_{k=0}^{\floor{n/2}}R_{n,k}t^k(1+t)^{n-2k}$ 
	where 
	$R_{n,k}$ satisfies the following recurrence:
	$R_{n,k}=(2k+1)R_{n-1,k}+4(n+1-2k)R_{n-1,k-1}$. We can take  
	$c_{n,k}=(2k+1)$ and $d_{n,k}=4(n+1-2k)$ to get the 
	log-concavity of $R_{n,k}$.

\end{enumerate}

\section{Open Problems}
\label{open_problems}

In this Section, we raise some questions and make some interesting conjectures. Define  $A_{n,k}^+$ and $A_{n,k}^-$ to be the number of 
permutations with $k$ descents in $\AAA_n$ and $\SSS_n - \AAA_n$ 
respectively. Based on data, we make the following conjecture about the 
sequences $A_{n,k}^+$ and $A_{n,k}^-$.

\begin{conjecture}
	\label{conj:descentfortypeA}
	For  positive integers $n$, the sequences 
	$(A_{n,k}^+)_{k=0}^{n-1}$ and $(A_{n,k}^-)_{k=0}^{n-1}$
	are strongly synchronised.
\end{conjecture} 

Conjecture \ref{conj:descentfortypeA} will show log-concavity 
of $A_{n,k}^+$ and $A_{n,k}^-$. Real-rootedness of
the polynomials 
$ A_n^+(t) = \sum_{\pi \in \AAA_n} t^{\des(\pi)}$
and $A_n^-(t) = \sum_{\pi \in \SSS_n - \AAA_n} t^{\des(\pi)}$ was conjectured by Dey and 
Sivasubramanian \cite[Conjecture 48]{siva-dey-gamma_positive_descents_alt_group_ejc} when 
$n \equiv 0,1 \mod 4$ and  extended by 
Fulman, Kim, Lee and Petersen \cite[Conjecture 1.3]{fulman-kim-lee-petersendescentcentrallimit} for all $n$.

Though descents and excedances are not equidistributed over $\AAA_n$, 
they seem to be strongly synchronised over $\AAA_n$ for all $n$, that is,

\begin{conjecture}
	\label{conj:descentandexcedancefortypeA}
	For positive integers $n$, the sequences 
	$(A_{n,k}^+)_{k=0}^{n-1}$ and $(P_{n,k})_{k=0}^{n-1}$
	are strongly synchronised. Similarly, 
	the sequences 
	$(A_{n,k}^-)_{k=0}^{n-1}$ and $(Q_{n,k})_{k=0}^{n-1}$
	are strongly synchronised. 
\end{conjecture}

\begin{problem}
	\label{prob:combinatorialproof}
	It would be very interesting to find  combinatorial proofs of 
	Theorem
	\ref{thm:MainresultforevenlogconcaveTypeA}, 
	Theorem \ref{thm:MainresultforevenlogconcaveTypeB}, 
	Conjecture \ref{conj:descentfortypeA}  
	and Conjecture \ref{conj:descentandexcedancefortypeA}. 
	B{\'o}na and Ehrenborg in \cite{bonaehrenborglogconcavity} have given 
	a combinatorial proof of log-concavity of $A_{n,k}$ but the coefficients 
	$A^+_{n,k}$ and $A^-_{n,k}$ are not even ratio-alternating. 
	Hence, this argument directly does not prove strong synchronisation 
	of $A_{n,k}^+$ and $A_{n,k}^-$. 
	\end{problem}

\section*{Acknowledgements}

The author would like to thank his advisor Sivaramakrishnan Sivasubramanian for all the insightful discussions and comments during the preparation of the paper. The author also thanks Subhajit Ghosh, Venkitesh Iyer and Brahadeesh Sankarnarayanan for some suggestions during the later phase of this work. The author also acknowledges funding from CSIR-SPM fellowship. 


\bibliographystyle{acm}
\bibliography{mainstrongsynchro.bib}

\end{document}